\documentclass[]{interact}

\usepackage[numbers,sort&compress]{natbib}
\bibpunct[, ]{[}{]}{,}{n}{,}{,}
\renewcommand\bibfont{\fontsize{10}{12}\selectfont}
\makeatletter
\def\NAT@def@citea{\def\@citea{\NAT@separator}}
\makeatother

\theoremstyle{plain}
\newtheorem{theorem}{Theorem}[section]

\newtheorem{corollary}[theorem]{Corollary}
\newtheorem{proposition}[theorem]{Proposition}

\theoremstyle{definition}
\newtheorem{definition}[theorem]{Definition}
\newtheorem{example}[theorem]{Example}

\theoremstyle{remark}
\newtheorem{remark}{Remark}

\newcommand{\R}{\mathbb R}
\newcommand{\N}{\mathbb N}

\newcommand{\Q}{\mathbb Q}

\newcommand{\Usfer}{{\mathbb S}}

\newcommand{\dom}{{\rm dom}\, }
\newcommand{\dmn}{{\rm dim}\, }
\newcommand{\gph}{{\rm gph}\,}
\newcommand{\krn}{{\rm ker}\, }
\newcommand{\img}{{\rm im}\, }
\newcommand{\id}{{\rm id}\, }    

\newcommand{\nullv}{\mathbf{0}}

\newcommand{\conv}{{\rm conv}\, }
\newcommand{\stardif}{\hbox{$\,{*\over {}}\,$}}

\newcommand{\sgn}{{\rm sgn}\, }

\newcommand{\Uplim}{{\rm Limsup}}

\newcommand{\Matmxn}{{\mathcal M}_{m\times n}(\R)}    

\newcommand{\Lin}{{\mathcal L}}    
\newcommand{\CC}{{\mathcal K}}     
\newcommand{\Sub}{{\mathcal S}}    
\newcommand{\DCH}{\mathcal{DS}}  
\newcommand{\Ccond}{({\mathfrak{C}})}    

\newcommand{\Rsubd}{\widehat{\partial}}    
\newcommand{\Bsubd}{{\partial}_M}    

\newcommand{\injg}[1]{{\rm inj}(#1)}   

\newcommand{\Fder}[2]{{\rm D}#1(#2)}   

\newcommand{\Jmat}[2]{{\rm J}#1(#2)}   

\newcommand{\Clder}[2]{{\rm \partial}^\circ#1(#2)}   

\newcommand{\sqdcoder}[2]{\widetilde{\rm D}^*#1(#2)}   

\newcommand{\starcoder}[2]{\widetilde{\rm D}^{\stardif}#1(#2)}   

\newcommand{\ball}[2]{{\rm B}[#1; #2]}      
\newcommand{\oball}[2]{{\rm B}(#1; #2)}   

\newcommand{\dist}[2]{{\rm dist}\left(#1,#2\right)}

\newcommand{\supp}[2]{{\varsigma}\left(#1;#2\right)}   

\newcommand{\RNcone}[2]{\widehat{{\rm N}}(#1;#2)}

\newcommand{\Rcoder}[2]{\widehat{{\rm D}}^*#1(#2)}   

\newcommand{\lip}[2]{{\rm lip}(#1;#2)}

\newcommand{\cov}[2]{{\rm cov}(#1;#2)}   

\newcommand{\lop}[2]{{\rm lop}(#1;#2)}    

\newcommand{\reg}[2]{{\rm reg}(#1;#2)}    

\newcommand{\inj}[2]{{\rm inj}(#1;#2)}   

\newcommand{\Apr}[3]{{\rm A}#1(#2;#3)}   

\begin{document}


\title{On some generalizations of Hadamard's inversion theorem beyond differentiability}

\author{
\name{A. Uderzo\textsuperscript{a}\thanks{CONTACT A. Uderzo. Email: amos.uderzo@unimib.it}}
\affil{\textsuperscript{a} Dipartimento di Matematica e Applicazioni, Universit\`a
di Milano-Bicocca, Milano, Italy}
}

\maketitle

\begin{abstract}
A recognized trend of research investigates generalizations of the Hadamard's
inversion theorem to functions that may fail to be differentiable.
In this vein, the present paper explores some consequences of a recent
result about the existence of global Lipschitz continuous inverse by
translating its metric assumptions in terms of nonsmooth analysis constructions.
This exploration focuses on continuous, but possibly not locally Lipschitz
mappings, acting in finite-dimensional Euclidean spaces.
As a result, sufficient conditions for global invertibility are formulated
by means of strict estimators, $\stardif$-difference of convex compacta
and regular/basic coderivatives. These conditions qualify the global inverse as
a Lipschitz continuous mapping and provide quantitative estimates of its
Lipschitz constant in terms of the above constructions.
\end{abstract}

\begin{keywords}
Global inversion; linear openness; metric injectivity; strong regularity;
strict $\mu$-estimator; difference of convex compacta; coderivatives.
\end{keywords}

\section{Introduction}

The celebrated Hadamard's inversion theorem establishes a connection between the global
invertibility of a mapping and certain differential properties of it.
Global as well as local invertibility is a key property that mappings may have,
widely employed in various contexts often in synergy with more structured properties
(see, for instance, \cite{Elcr82}). Such a property can be considered in highly
abstract settings, because in its essence it is
free from references to specific structures, such as the topological, the
metric and the differential ones. This freedom however may result in a poor
interplay with the latter ones, provided that they are at disposal, as
illustrated by the example below.

\begin{example}
Let $f:\R\longrightarrow\R$ be defined by
$$
  f(x)=x\cdot\chi_{{}_{\Q}}(x)-x\cdot\chi_{{}_{\R\backslash\Q}}(x)=
  \left\{\begin{array}{ll}
  x, & \quad\ \hbox{ if } x\in\Q, \\
  \\
  -x, & \quad\ \hbox{ if } x\in\R\backslash\Q.
  \end{array}
  \right.
$$
If $\R$ is equipped with its usual Euclidean space structure, it is
clear that $f$ is nowhere differentiable, while is continuous only
at $0$, yet it is globally invertible, namely there exists
$f^{-1}:\R\longrightarrow\R$ (which is given, in the present case,
by $f$ itself).
\end{example}

That said, it must be recognized that the study of invertibility
conditions took benefits from tools of differential calculus and linear
algebra.
It is well known indeed that if $f:\R^n\longrightarrow\R^n$
is continuously differentiable in a neighbourhood of a point $\bar x\in\R^n$
and its (Fr\'echet) derivative $\Fder{f}{\bar x}:\R^n\longrightarrow\R^n$
at $\bar x$ is invertible (the Jacobian matrix $\Jmat{f}{\bar x}$ representing the linear
mapping $\Fder{f}{\bar x}$ is nonsingular), then $f$ admits a local
inverse, which is continuously differentiable in a neighbourhood of
$f(\bar x)$. This classic result was subsequently generalized to contexts of possible
lack of differentiability in several ways (see, among the others, \cite{Clar76},
\cite[Theorem 1E.3]{DonRoc14}, \cite[Theorem 3.3.1]{JeyLuc08}).

Since derivatives (or their surrogates) provide only a local approximation of a mapping,
the above result is expected to hold only locally, as it actually does.
The great advance made with the Hadamard's theorem consists in providing
a sufficient condition for the global invertibility of a mapping in terms of
derivatives. Its statement is recalled below.

\begin{theorem}[Hadamard's inversion theorem]     \label{thm:Hadamard}
Given a mapping $f:\R^n\longrightarrow\R^n$, suppose that:

\begin{itemize}

\item[(i)] $f\in {\rm C}^1(\R^n)$;

\item[(ii)] for every $x\in\R^n$ there exists $\Fder{f}{x}^{-1}$;

\item[(iii)] there exists $\kappa>0$ such that
$\sup_{x\in\R^n}\|\Fder{f}{x}^{-1}\|_\Lin\le\kappa$.

\end{itemize}

\noindent Then $f$ is a global diffeomorphism, namely:

\begin{itemize}

\item[(t)] $\exists f^{-1}:\R^n\longrightarrow\R^n$;

 \item[(tt)] $f^{-1}$ is differentiable on $\R^n$.

\end{itemize}
\end{theorem}

A further step in providing more general sufficient conditions for
the global invertibility of a mapping has been done in the early 80s
with the appearance of \cite{Pour82}, when the unnecessary smoothness assumption
on $f$ was dropped out. Its role was replaced by Lipschitz continuity and
a proper condition on the Clarke subdifferential of the mapping.
Recall that a mapping $f:\R^n\longrightarrow\R^m$ is said to be
Lipschitz continuous with constant $\ell>0$ in a subset $U\subseteq\R^n$
if
\begin{equation}   \label{in:Lipcondef}
  \|f(x_1)-f(x_2)\|\le \ell\|x_1-x_2\|,\quad\forall x_1,\, x_2\in U.
\end{equation}
Whenever inequality (\ref{in:Lipcondef}) holds in a neighbourhood
of a given point $\bar x$, $f$ is said to be locally Lipschitz
around $\bar x$. In such an event,
the infimum over all $\ell$, for which there exists a neighbourhood $U$
of $\bar x\in\R^n$ such that (\ref{in:Lipcondef}) holds, will be denoted
throughout the paper by $\lip{f}{\bar x}$.
According to the Rademacher's theorem, a mapping $f$ which is locally Lipschitz
around $\bar x$ turns out to be Fr\'echet differentiable in a (Lebesgue)
full measure subset of a neighbourhood of $\bar x$. Thus, it is possible
to define
$$
  \Clder{f}{\bar x}=\conv\{M\in\Matmxn\ |\ M=\lim_{k\to\infty}\Jmat{f}{x_k},\
  \hbox{ for some }\ \{x_k\}_{k\in\N},\ x_k\not\in\Omega_f\},
$$
where $\Omega_f$ denotes the null measure subset of a proper neighbourhood
of $\bar x$ containing all those points at which $f$ is not Fr\'echet
differentiable and $\conv S$ denotes the convex closure of the set $S$.

\begin{theorem}[Pourciau's inversion theorem]    \label{thm:Pourciau}
Given a mapping $f:\R^n\longrightarrow\R^n$, suppose that:

\begin{itemize}

\item[(i)] $f$ is locally Lipschitz around each point $x\in\R^n$;

\item[(ii)] for every $x\in\R^n$ and for every $M\in \Clder{f}{x}$,
there exists $M^{-1}$;

\item[(iii)] there exists $\kappa>0$ such that
$\sup_{x\in\R^n}\sup_{M\in \Clder{f}{x}}\|M^{-1}\|\le\kappa$.

\end{itemize}

\noindent Then $f$ is a global homeomorphism, namely:

\begin{itemize}

\item[(t)] $\exists f^{-1}:\R^n\longrightarrow\R^n$;

 \item[(tt)] $f^{-1}$ is Lipschitz continuous on $\R^n$ with constant $\kappa$.

\end{itemize}
\end{theorem}

Since then, the study of conditions for global invertibility remained
in the agenda of several groups of researchers.
A recent trend of investigation is exploring the potential of concepts
emerged in modern variational analysis properly combined with nonsmooth
analysis constructions in deriving new generalizations of Hadamard's
theorem (see, among the others, \cite{AruZhu19,Gutu22,GutJar19,JaLaMa19}).

The purpose of the present paper is to carry on such a trend,
sheding light on some possible combinations, which seem to have
not yet gleaned. More precisely, following existent approaches,
this is done by analyzing local invertibility through metric regularity,
metric injectivity and strong regularity and their stability properties
under perturbation. Such an approach leads to replace generalized
derivatives with looser approximation concepts (strict $\mu$-estimators
or coderivatives),
which can be exploited in formulating general conditions for global
invertibility.

The contents of the paper are arranged as follows.
In Section \ref{Sect:2} a general scheme for deriving conditions
for global inversion of possibly nonsmooth mappings is presented.
It stems from a recent global inversion result relying on
variational analysis concepts of purely metric nature.
It leads to formulate two generalizations of Hadamard's theorem,
where derivatives are replaced by strict estimators.
In Section \ref{Sect:3} a result obtained in the preceding section
is specialized to the mappings admitting strict estimators, which are
dually representable by means of pairs of convex compacta.
Section \ref{Sect:4} comes back to a more general setting,
providing a global invertibility condition, where the role of
derivatives is played by regular/basic coderivatives.
Conclusions are gathered in Section \ref{Sect:5}.

The notations in use throughout the paper are standard.
In an Euclidean space $(\R^n,\|\cdot\|)$, $\ball{x}{r}$ denotes
the closed ball, with center at $x$ and radius $r$, while $\oball{x}{r}$
stands for the open one. The unit sphere centered at the null vector
$\nullv\in\R^n$ is indicated by $\Usfer$. The metric enlargement
with radius $r$ of a set $S\subseteq\R^n$ is denoted by $\oball{S}{r}=
\{x\in\R^n\ |\ \dist{x}{S}<r\}$, where $\dist{x}{S}=\inf_{z\in S}\|z-x\|$
is the distance of $x$ from $S$.
$\Sub(\R^n)$ denotes the cone of all sublinear (i.e. p.h. and convex)
functions defined on $\R^n$, whereas $\DCH(\R^n)=\Sub(\R^n)-\Sub(\R^n)$.
$\CC(\R^n)$ is the class of all nonempty convex compact subsets of $\R^n$.
By $\supp{\cdot}{S}$ the support function associated to the subset
$S\subseteq\R^n$ is denoted. If $\varphi:\R^n\longrightarrow\R\cup\{\pm\infty\}$
is a convex function, $\partial\varphi(x)$ stands for its subdifferential
at $x\in\dom\varphi=\varphi^{-1}(\R)$ in the sense of convex analysis,
whereas if $\varphi$ is concave
$\partial^+\varphi(x)$ is the superdifferential of $\varphi$ at $x$.
If $l:\R^n\longrightarrow\R^m$ is a linear mapping, $\krn l$, $\img l$
and $\|l\|_\Lin$ indicate the kernel, the range and the operator norm
of $l$, respectively. By $\|l\|_\Lin$ the operator norm of $l$ is
denoted.
The acronyms p.h. and u.s.c. stand for positively homogeneous and
upper semicontinuous, respectively.
Further notations will be introduced subsequently, contextually to
their use.

\vskip1cm


\section{A general scheme for global inversion without derivatives} \label{Sect:2}

The analysis here exposed start with a recent global invertibility
result established in \cite[Theorem 2.2]{AruZhu19}, which relies on
a suitable metric behaviour of mappings, disregarding
their differential properties. Such a metric behaviour deals with
a covering property formalized as follows.

\begin{definition}    \label{def:acover}
Given $\alpha>0$ and $x\in\R^n$, a mapping $f:\R^n\longrightarrow\R^m$
is said to be {\it $\alpha$-covering at $x$} if for every $\epsilon>0$
there exists $r\in (0,\epsilon]$ such that
\begin{equation}     \label{in:defacov}
  \ball{f(x)}{\alpha r}\subseteq f(\ball{x}{r}).
\end{equation}
The value
$$
   \cov{f}{\bar x}=\sup\{\alpha>0\ |\ \hbox{ for which
   (\ref{in:defacov}) holds}\,\}
$$
is called {\it covering bound} of $f$ at $x$.
If inclusion (\ref{in:defacov}) holds true for all $x\in\R^n$ and
$r\ge 0$, with the same value of $\alpha$, then $f$ will be called
{\it $\alpha$-covering}.
\end{definition}

In what follows, a mapping $f:\R^n\longrightarrow\R^n$ is said to
be locally injective provided that every $x\in\R^n$ admits a
neighbourhood in which $f$ is injective.

\begin{theorem}[\cite{AruZhu19}]    \label{thm:AruZhu}
Given a mapping $f:\R^n\longrightarrow\R^n$, suppose that:

\begin{itemize}

\item[(i)] $f$ is continuous on $\R^n$;

\item[(ii)] $f$ is locally injective;

\item[(iii)] $f$ is $\alpha$-covering at each $x\in\R^n$.

\end{itemize}

\noindent Then

\begin{itemize}

\item[(t)] $\exists f^{-1}:\R^n\longrightarrow\R^n$;

\item[(tt)] $f^{-1}$ is Lipschitz continuous on $\R^n$
with constant $\alpha^{-1}$.

\end{itemize}
\end{theorem}

It is worth noticing that Theorem \ref{thm:AruZhu} avoids
any reference to (traditional or generalized) differential
properties of mappings, while focusing on basic, yet
more than set-theoretic (topological and metric) properties,
namely continuity, injectivity and metric covering.
A related feature is that its thesis is not only qualitative
($f$ is a homeomorphism), but it provides quantitative estimates
for the Lipschitz constant of $f^{-1}$.
In the above theorem, a key role is played by hypothesis (iii),
which is indeed a metric surjection assumption. The property of
being $\alpha$-covering at a given point is a weaker (point-based)
variant of the linear openness (a.k.a. openness at a linear rate),
which is well-known and largely employed in variational analysis.
Recall that a mapping
$f:\R^n\longrightarrow\R^m$ is said to be linearly open
around a point $\bar x\in\R^n$ if there exists a constant $\alpha>0$
together with neighbourhoods $U$ of $\bar x$ and $V$ of $f(\bar x)$
such that
\begin{equation}     \label{in:lopdef}
  \oball{f(x)}{\alpha r}\cap V\subseteq f\left(\oball{x}{r}\right),
  \quad\forall x\in\ U,\, \forall r>0.
\end{equation}
The quantity
$$
   \lop{f}{\bar x}=\sup\{\alpha>0\ |\ \exists U,\ V \hbox{ for which
   (\ref{in:lopdef}) holds}\,\}
$$
is called (exact) linear openness bound of $f$ around $\bar x$
(see \cite{DonRoc14,Mord06,Mord18}).

\begin{remark}     \label{rem:lopenacov}
If a continuous mapping $f:\R^n\longrightarrow\R^m$ is linearly
open around $\bar x$, then for any $\alpha<\lop{f}{\bar x}$ it is
also $\alpha$-covering at $\bar x$ in the sense of Definition \ref{def:acover},
and it is true that $\lop{f}{\bar x}\le\cov{f}{\bar x}$.
On the other hand, if the property of being $\alpha$-covering
holds true at each point of $\R^n$ with the same $\alpha>0$, then it implies linear
openness around each point $x$ of $\R^n$, with $\lop{f}{x}\ge\alpha$.
Indeed, as done at the beginning of the proof of \cite[Theorem 2.2]{AruZhu19}
by means of a Caristi-like condition (see \cite[Lemma 4.2]{AruZhu19}),
it is possible to prove that, for a continuous mapping, $\alpha$-covering
at each point of $\R^n$ actually implies $\alpha$-covering.
\end{remark}

The concept of linear openness turned out to be highly fruitful
for the development of several topics in variational and nonlinear analysis,
in particular under its equivalent reformulation in terms of
metric regularity. It is indeed known that $f$ is linearly open
around $\bar x$ iff there exists $\kappa>0$ together with a neighbourhood
$U$ of $\bar x$ and $V$ of $f(\bar x)$ such that
$$
  \dist{x}{f^{-1}(y)}\le\kappa\|y-f(x)\|,\quad
  \forall (x,y)\in U\times V.
$$
The infimum over all $\kappa$ for which there exist $U$ and $V$
satisfying the above inequality is called (exact) metric regularity bound
of $f$ around $\bar x$ and will be denoted throughout the paper by
$\reg{f}{\bar x}$. In other terms, both linear openness and metric regularity
are seemingly different manifestations of the same metric
behaviour of a mapping. The linear openness bound and the metric
regularity bound are known to be intertwined by the sharp relation
\begin{equation}     \label{eq:lopregrel}
  \lop{f}{\bar x}\cdot \reg{f}{\bar x}=1,
\end{equation}
with the conventions that $0\cdot\infty=1$, while
$\lop{f}{\bar x}=0$ and $\reg{f}{\bar x}=\infty$
indicate the lack of the related property
(see \cite[Theorem 3E.6]{DonRoc14}, \cite[Theorem 3.2]{Mord18}).

In its essence, differentiability of a mapping amounts to be suitable
for being locally approximated by linear mappings. Nonsmooth analysis
went beyond this property, by admitting relaxed forms of
approximation (via directional restrictions or metric estimates)
and by removing the linearity of the approximating terms.
Among several possible ways to accomplish this (see \cite{DemRub95,JeyLuc08,Mord06,PalRol97})
in \cite[Chapter 1.5]{DonRoc14} this approach resulted in the following notion.

\begin{definition}[Strict $\mu$-estimator]    \label{def:mestim}
Consider a mapping $f:\R^n\longrightarrow\R^m$, a point $\bar x\in\R^n$
and $\mu\ge 0$. A function $h_{\bar x}:\R^n\longrightarrow\R^m$ is said
to be a {\it strict $\mu$-estimator} of $f$ at $\bar x$ if
\begin{itemize}

\item[(i)] $h_{\bar x}(\bar x)=f(\bar x)$;

\item[(ii)] $\lip{f-h_{\bar x}}{\bar x}\le\mu$.

\end{itemize}
\end{definition}

More than generalized derivatives, strict $\mu$-estimators provide
a scheme for building various kind of controlled metric perturbations
of a given mapping. In the example below, some specific implementations
of this concept are presented.

\begin{example}[Strict first-order approximation]    \label{ex:strapprox}
Following \cite[Chapter 1.5]{DonRoc14}, given a mapping $f:\R^n\longrightarrow\R^m$
and a point $\bar x\in\R^n$, a mapping $\Apr{f}{\bar x}{\cdot}:\R^n\longrightarrow\R^m$
is said to be a strict first-order approximation of $f$ at $\bar x$ if
$\Apr{f}{\bar x}{\bar x}=f(\bar x)$ and
\begin{equation}     \label{eq:defstrapprox}
  \lip{f-\Apr{f}{\bar x}{\cdot}}{\bar x}=0.
\end{equation}
These conditions obviously imply that any strict first-order
approximation of $f$ at $\bar x$ is, in particular, a strict $\mu$-estimator
of $f$ at $\bar x$, for every $\mu\ge 0$. It is worth observing that
the axiomatic notion of strict  first-order approximation can cover several
classic and generalized (less classic) situations of differentiability.
For instance, if $\Apr{f}{\bar x}{\cdot}$
can be taken in the class of all affine mappings, then, whenever $f$ is
strictly differentiable at $\bar x$, by setting
$$
  \Apr{f}{\bar x}{x}=f(\bar x)+\Fder{f}{\bar x}[x-\bar x]
$$
one sees that its first-order expansion at $\bar x$, expressed by
its Fr\'echet derivative, falls in the general scheme of strict
first-order approximation.

Another specific instance of strict first-order approximation
is given by the notion of strong Bouligand derivative. According
to \cite[Definition 3.1.2]{FacPan03}, $f:\R^n\longrightarrow\R^m$
is strongly Bouligand differentiable at $\bar x$ if it is locally
Lipschitz around $\bar x$ and condition (\ref{eq:defstrapprox})
holds with
$$
  \Apr{f}{\bar x}{x}=f(\bar x)+f'(\bar x;x-\bar x),
$$
where $f'(\bar x;v)$ denotes the directional derivative of $f$
at $\bar x$, in the direction $v\in\R^n$.

Let us further mention that
strict first-order approximations appeared in a more general setting
(normed linear spaces) already in \cite{Robi91} under the name of
strong approximations.
\end{example}

\begin{remark}    \label{rem:strmuestpro}
(i) From Definition \ref{def:mestim}(ii) it follows that if $h_{\bar x}$
(resp. $f$) is locally Lipschitz around $\bar x$, then $f$ (resp.
$h_{\bar x}$) inherits the same property around $\bar x$.
Nevertheless, it may happen that mappings failing to be locally
Lipschitz around a reference point do admit estimators at that
point, of course affected by the same lack of Lipschitz continuity
(see Example \ref{ex:gloinvnotlip}).

(ii) For the subject under investigation, the question of existence
of $\mu$-estimators for a given mapping at a given point becomes
a relevant issue. Example \ref{ex:strapprox} and the point (i) of
the present remark should offer perspectives for addressing this
question. Obviously, whenever $f$ is locally Lipschitz around
$\bar x$, it admits $h_{\bar x}\equiv 0$ as a $\mu$-estimator at $\bar x$,
for every $\mu<\lip{f}{\bar x}$. On the other hand, in any case
$f$ admits $h_{\bar x}=f$ as a $\mu$-estimator at $\bar x$,
for every $\mu\ge 0$ and $\bar x\in\R^n$. The contexts of employment
should determine the convenience of choices to be made, through specific
requirements on the class $\{h_x:\ x\in\R^n\}$.
\end{remark}

The next proposition justifies the employment of $\mu$-estimators
in the current approach by asserting the stability behaviour of linear
openness/metric regularity in the presence of additive Lipschitz
perturbations, which was independently observed by A.A. Milyutin and
S.M. Robinson. Actually, it captures the quantitative stability
phenomenon behind many recent extensions of the celebrated
Lyusternik-Graves theorem. As presented below, it is a reformulation of
\cite[Theorem 3F.1]{DonRoc14}, adapted in the light of the relation
(\ref{eq:lopregrel}).

\begin{proposition}[Linear openness via strict estimators]     \label{pro:lopenest}
Let $f:\R^n\longrightarrow\R^m$ be a given mapping, let $\bar x\in\R^n$,
and $\mu\ge 0$. Suppose that:

\begin{itemize}

\item[(i)] $f$ admits a strict $\mu$-estimator $h_{\bar x}:
\R^n\longrightarrow\R^m$ at $\bar x$;

\item[(ii)] $\lop{h_{\bar x}}{\bar x}>\mu$.

\end{itemize}

Then, $f$ is linearly open around $\bar x$, with $\lop{f}{\bar x}\ge
\lop{h_{\bar x}}{\bar x}-\mu$. In particular, $f$ is $\alpha$-covering
at $\bar x$ with any $\alpha<\lop{h_{\bar x}}{\bar x}-\mu$.
\end{proposition}

The impact of Proposition \ref{pro:lopenest} becomes clear when
studying linear openness of $h_{\bar x}$ is easier than studying
the same property of $f$, by virtue of the specific features of
$h_{\bar x}$ (e.g. linearity, positive homogeneity, some form
of convexity property for vector-valued mappings, and so on).


Together with surjection properties, an enhanced (again, metric) form of injectivity
will be employed in the sequel. To the best of the author's knowledge,
such a property was first formalized as follows in \cite{Gutu22},
wherefrom the terminology is borrowed.

\begin{definition}    \label{def:minjproper}
A mapping $f:\R^n\longrightarrow\R^m$ is said to be
{\it metrically injective} around $\bar x\in\R^n$ if there exist constants
$\beta,\, \delta>0$ such that
\begin{equation}     \label{in:defminj}
  \|f(x_1)-f(x_2)\|\ge\beta \|x_1-x_2\|,\quad\forall x_1,\, x_2\in
  \ball{\bar x}{\delta}.
\end{equation}
The value
$$
   \inj{f}{\bar x}=\sup\{\beta>0\ |\ \exists\delta>0 \hbox{ for which
   (\ref{in:defminj}) holds}\,\}
$$
is called (exact) {\it metric injection bound} of $f$ around $\bar x$.
\end{definition}

\begin{remark}
(i) It is plain to see that, whenever a mapping $f$ is metrically injective
around a point, it must be locally injective around that point,
whereas the converse implication fails to be true, in general
(see Example \ref{ex:minjlinmap}).

(ii) Notice that the metric injection bound of a mapping $f$ around
$\bar x$ can be expressed in terms of displacement rate of the
values taken by $f$ as follows
$$
  \inj{f}{\bar x}=\liminf_{x_1,\, x_2\to\bar x\atop x_1\ne x_2}
  {\|f(x_1)-f(x_2)\|\over\|x_1-x_2\|}.
$$
\end{remark}

\begin{example}[Characterization of metric injectivity for linear
mappings] Let $l:\R^n\longrightarrow\R^m$ be a linear mapping.
As a preliminary remark about metric injectivity of linear mappings,
notice that, because this property implies mere injectivity and hence the relation
$\dmn\krn l+\dmn\img l=0+\dmn\img l=n$ must be true, then a necessary
dimensional condition for metric injectivity of linear mapping to hold
is that $m\ge n$.

Then, the following are equivalent:

\begin{itemize}

\item[(i)] $l$ is metrically injective all over $\R^n$
(i.e. inequality (\ref{in:defminj}) holds with $\delta=+\infty$);

\item[(ii)] $l$ is metrically injective around $\nullv$;

\item[(iii)] $\dist{\nullv}{l[\Usfer]}>0$.

\end{itemize}

Implication (i) $\Rightarrow$ (ii) is obvious.

(ii) $\Rightarrow$ (iii): Since, in particular, for some $\beta,\,\delta>0$, it is
$\|l[x]\|\ge\beta\|x\|$ for every $x\in\ball{\nullv}{\delta}$, one obtains
by linearity
\begin{equation}    \label{in:lbetaUsfer}
  \|l[u]\|\ge\beta,\quad\forall u\in\Usfer,
\end{equation}
whence, passing to the infimum over $\Usfer$, it immediately follows
$$
  \dist{\nullv}{l[\Usfer]}\ge\beta>0.
$$

(iii) $\Rightarrow$ (i): Set $\beta=\dist{\nullv}{l[\Usfer]}$. Then
inequality (\ref{in:lbetaUsfer}) must be true. Thus, by taking
an arbitrary pair $x_1,\, x_2\in\R^n$, with $x_1\ne x_2$, one can
write
$$
 \left\|l\left[{x_1-x_2\over \|x_1-x_2\|}\right]\right\|\ge\beta.
$$
From the last inequality, one obtains by linearity
\begin{equation}    \label{in:minjglol}
  \|l[x_1]-l[x_2]\|\ge\beta\|x_1-x_2\|,\quad\forall x_1,\, x_2\in\R^n.
\end{equation}
The reader should notice that, by exploiting the above inequalities,
one can deduce also that it holds
$$
  \inj{l}{x}=
  \injg{l}=\sup\{\beta>0\ |\ (\ref{in:minjglol}) \hbox{ holds }\}
  =\dist{\nullv}{l[\Usfer]},\ \forall x\in\R^n.
$$
In operator theory, the quantity $\dist{\nullv}{l[\Usfer]}$
is connected with the Banach constant of a bounded linear
operator between Banach spaces. More precisely, if referred to
the adjoint operator to $l$, here denoted by $l^*$, the Banach constant
quantifies the (global) linear openness of an onto operator,
i.e. $\lop{l}{\nullv}=\lop{l}{x}=\cov{l}{\nullv}=
\dist{\nullv}{l^*[\Usfer]}$ (see, for instance,
\cite[Corollary 1.58]{Mord06}). This notion appears
in \cite{JaLaMa19} under the name of surjectivity index.
The link between injectivity of $l$ and openness of $l^*$
is well known in operator theory and it is summarized by
the condition on the annihilator of the kernel of a regular
operator: $(\krn l)^\bot=\img l^*$, where $S^\bot$ denotes the
annihilator of a subspace $S$ of a Banach space.
\end{example}

\begin{example}[P.h. scalar functions]
Let $h:\R\longrightarrow\R$ be a p.h. continuous function, taking
the form
$$
  h(x)=\left\{\begin{array}{ll}
  \theta_+x, & \quad\ \hbox{ if } x\ge 0, \\
  \\
  \theta_- x, & \quad\ \hbox{ if } x<0,
  \end{array}
  \right.
$$
for proper $\theta_+,\, \theta_-\in\R$.
Then $h$ is metrically injective around $0$ iff it is locally injective
around the same point. This happens iff $h$ is injective all over $\R$,
namely iff it is strictly monotone (increasing or decreasing), or,
equivalently, iff
\begin{equation}     \label{in:monotthetacond}
  \theta_-\cdot\theta_+>0.
\end{equation}
Upon condition (\ref{in:monotthetacond}), if $0<\theta_-\le\theta_+$,
$h$ turns out to be sublinear and $0\not\in\partial h(0)=[\theta_-,\theta_+]$.
Upon condition (\ref{in:monotthetacond}), if $0<\theta_+\le\theta_-$,
$h$ turns out to be superlinear and $0\not\in\partial^+ h(0)=[\theta_+,\theta_-]$.
Again under condition (\ref{in:monotthetacond}), if $\theta_-\le\theta_+<0$,
$h$ is sublinear and $0\not\in\partial h(0)=[\theta_-,\theta_+]$, whereas
if $\theta_+\le\theta_-<0$, $h$ is superlinear and $0\not\in\partial^+
h(0)=[\theta_+,\theta_-]$. In any case,
$$
  \inj{h}{0}=\min\{|\theta_+|,|\theta_-|\}=
  \left\{\begin{array}{ll}
  \dist{0}{\partial h(0)}, & \quad\ \hbox{ if $h$ is sublinear}, \\
  \\
  \dist{0}{\partial^+ h(0)}, & \quad\ \hbox{ if $h$ is superlinear}.
  \end{array}
  \right.
$$

\end{example}

The next proposition establishes a useful stability behaviour of metric
injectivity under additive Lipschitz perturbations.

\begin{proposition}    \label{pro:minjsta}
Given a mapping $g:\R^n\longrightarrow\R^m$, suppose that:

\begin{itemize}

\item[(i)] $g$ is metrically injective around $\bar x\in\R^n$;

\item[(ii)] $h$ is locally Lipschitz around $\bar x$, with
$\lip{h}{\bar x}<\inj{g}{\bar x}$.

\end{itemize}
Then, the mapping $g+h$ is metrically injective, with
\begin{equation}    \label{in:minjsum}
  \inj{g+h}{\bar x}\ge \inj{g}{\bar x}-\lip{h}{\bar x}.
\end{equation}
\end{proposition}

\begin{proof}
It is clear that, by the triangle inequality, it holds
\begin{eqnarray*}
  \|g(x_1)-g(x_2)\| &\le& \|g(x_1)+h(x_1)-g(x_2)-h(x_2)\| \\
  & + &\|h(x_1)-h(x_2)\|, \quad\forall x_1,\, x_2\in\R^n,
\end{eqnarray*}
whence it follows
\begin{equation}   \label{in:f+gge}
  \|(g+h)(x_1)-(g+h)(x_2)\|\ge \|g(x_1)-g(x_2)\|-\|h(x_1)-h(x_2)\|,
  \quad\forall x_1,\, x_2\in\R^n.
\end{equation}
By hypothesis (i), taken an arbitrary $\beta\in (0,\inj{g}{\bar x})$
there exists $\delta_\beta>0$ such that
\begin{equation}    \label{in:fminj}
     \|g(x_1)-g(x_2)\|\ge\beta\|x_1-x_2\|,\quad\forall
    x_1,\, x_2\in\ball{\bar x}{\delta_\beta}.
\end{equation}
By hypothesis (ii), taken an arbitrary $\ell>\lip{h}{\bar x}$
there exists $\delta_\ell>0$ such that
\begin{equation}  \label{in:hloclip}
     \|h(x_1)-h(x_2)\|\le \ell\|x_1-x_2\|,\quad\forall
    x_1,\, x_2\in\ball{\bar x}{\delta_\ell}.
\end{equation}
Thus, by taking $\delta=\min\{\delta_\beta,\, \delta_\ell\}$ and combining inequalities
(\ref{in:f+gge}), (\ref{in:fminj}), and (\ref{in:hloclip}), one obtains
$$
  \|(g+h)(x_1)-(g+h)(x_2)\|\ge (\beta-\ell)\|x_1-x_2\|,\quad\forall
    x_1,\, x_2\in\ball{\bar x}{\delta},
$$
which shows that $g+h$ is metrically injective around $\bar x$ and
$$
  \inj{g+h}{\bar x}\ge \beta-\ell.
$$
Since the last inequality is true for every $\beta<\inj{g}{\bar x}$
and $\ell>\lip{h}{\bar x}$, then by passing to the supremum over $\beta$
and the infimum over $\ell$, one gets inequality (\ref{in:minjsum}), thereby
completing the proof.
\end{proof}

In order for demonstrating that an analogous stability
behaviour does not hold for mere injectivity, one can consider the
following example.

\begin{example}     \label{ex:minjlinmap}
Let $f:\R\longrightarrow\R$ be given by $f(x)=x^3$. This function
is globally injective. Therefore, fixing $\bar x=0$, $f$ is injective
in particular in a neighbourhood of $0$, whereas it fails to be
metrically injective at the same point. As a perturbation terms
let us take the linear (and hence Lipschitz continuous) functions
$h_\zeta:\R\longrightarrow\R$, given by
$$
  h_\zeta(x)=-\zeta^2x,\qquad \zeta\in\R.
$$
For every $\delta>0$ one can find a function $h_\zeta$, with
$\zeta\in (0,\delta)$, such that the sum $f+h_\zeta$, namely
the function
$$
  x\mapsto f(x)+h_\zeta(x)=x(x-\zeta)(x+\zeta),
$$
is clearly not injective in $[-\delta,\delta]$. Notice that
$\lip{h_\zeta}{0}=\zeta^2$ can be chosen arbitrarily small.
\end{example}

The stability behaviour stated in Proposition \ref{pro:minjsta}
enables one to employ estimators for detecting metric injectivity,
in the same way as done for linear openness.

\begin{corollary}[Metric injectivity via strict estimators]    \label{cor:minjest}
Let $f:\R^n\longrightarrow\R^m$ be a mapping and let $\bar x\in\R^n$.
If $f$ admits a strict $\eta$-estimator $h_{\bar x}$ at $\bar x$ and
$\inj{h_{\bar x}}{\bar x}>\eta$,
then $f$ inherits from $h_{\bar x}$ the metric injectivity around
$\bar x$, with bound
$$
  \inj{f}{\bar x}\ge\inj{h_{\bar x}}{\bar x}-\eta.
$$
\end{corollary}

\begin{proof}
Since by hypothesis it is $\lip{f-h_{\bar x}}{\bar x}\le\eta$, it suffices to
apply Proposition \ref{pro:minjsta} with $g=h_{\bar x}$ and $h=f-h_{\bar x}$.
\end{proof}

On the base of the above constructions, the following condition for global
inversion can be established.

\begin{proposition}    \label{pro:gloinvest}
Let $f:\R^n\longrightarrow\R^n$ be a mapping and let $\mu\ge 0$.
Suppose that:

\begin{itemize}

\item[(i)] $f$ is continuous on $\R^n$;

\item[(ii)] for every $x\in\R^n$ $f$ admits a strict $\eta_x$-estimator $g_x$
such that $\inj{g_x}{x}>\eta_x$;

\item[(iii)] for every $x\in\R^n$ $f$ admits a strict $\mu$-estimator
$h_x$ and
$$
  \sigma_f=\inf_{x\in\R^n}\lop{h_x}{x}>\mu.
$$
\end{itemize}

\noindent Then,
\begin{itemize}

\item[(t)] $\exists f^{-1}:\R^n\longrightarrow\R^n$;

\item[(tt)] $f^{-1}$ is Lipschitz continuous on $\R^n$
with constant $(\sigma_f-\mu)^{-1}$.

\end{itemize}
\end{proposition}

\begin{proof}
By hypothesis (ii) and Corollary \ref{cor:minjest}, $f$ is metrically
injective at each point of $\R^n$ and hence locally injective.
Since for any $x\in\R^n$ $f$ admits a strict $\mu$-estimator $h_x$,
with $\lop{h_x}{x}>\mu$, according to Proposition \ref{pro:lopenest} $f$ is
$\alpha$-covering at $x$, with any $\alpha<\sigma_f-\mu$. As
a consequence, by recalling what observed in Remark \ref{rem:lopenacov},
$f$ is $\alpha$-covering with any $\alpha<\sigma_f-\mu$.
These facts mean that all the hypotheses of Theorem \ref{thm:AruZhu}
happen to be fulfilled. Accordingly, there exists $f^{-1}:\R^n\longrightarrow\R^n$,
which is Lipschitz continuous on $\R^n$ with constant
$$
  \alpha^{-1}>{1\over \sigma_f-\mu}.
$$
By arbitrariness of $\alpha<\sigma_f-\mu$ also assertion (tt)
in the thesis follows. This completes the proof.
\end{proof}

The aim of Proposition \ref{pro:gloinvest} is to define an interplay
between metric properties and approximation tools, which rules
viable methods for deriving conditions for global invertibility,
which are formulated
in terms of more specific nonsmooth analysis constructions.
In order to implement the above scheme, the next property may be of help.

\begin{definition}[Strong regularity]    \label{def:strreg}
A mapping $f:\R^n\longrightarrow\R^m$ is said to be {\it strongly metrically regular}
(henceforth, for short, {\it strongly regular}) around a point $\bar x\in\R^n$
if it has both the following properties:
\begin{itemize}

\item[(i)] $f$ is metrically regular (equivalently, linearly open)
around $\bar x$;

\item[(ii)] the set-valued mapping $f^{-1}:\R^m\rightrightarrows\R^n$
has a graphical localization around $(f(\bar x),\bar x)$, which is
nowhere multi-valued, i.e. there exist neighbourhoods $V$ of $f(\bar x)$
and $U$ of $\bar x$ such that
$$
  f^{-1}(y)\cap U=\{x\},\quad\forall y\in V.$$

\end{itemize}
\end{definition}

The above enhanced form of metric regularity can be equivalently reformulated
by postulating that the multifunction $f^{-1}$ admits a
single-valued localization $f^\sharp:V\longrightarrow\R^n$ around
$(f(\bar x),\bar x)$, which is locally Lipschitz around $f(\bar x)$,
with
\begin{equation}     \label{eq:lipreglop}
  \lip{f^\sharp}{f(\bar x)}=\reg{f}{\bar x}=\lop{f}{\bar x}^{-1}.
\end{equation}
(see \cite[Proposition 3G.1]{DonRoc14}).

For the purposes of the present paper, it is convenient to provide
the following characterization of strong regularity for continuous
mapping, which involves the notion of metric injection. To the best
of the author's knowledge, it remains still unnoticed within the variational
analysis literature.

\begin{proposition}  \label{pro:strrelchar}
Let $f:\R^n\longrightarrow\R^m$ be a mapping continuous at $\bar x\in\R^n$.
$f$ is strongly regular around $\bar x$ iff it is metrically regular and
metrically injective around $\bar x$. In such an event, it holds
\begin{equation}   \label{in:injloprel}
  \inj{f}{\bar x}\ge\lop{f}{\bar x}.
\end{equation}
\end{proposition}

\begin{proof}
Assume first that $f$ is metrically regular and metrically injective around
$\bar x$. Since $f$ is continuous at $\bar x$ and linearly open
around $\bar x$, one has in particular
\begin{equation}   \label{in:acoverfbarx}
 \ball{f(\bar x)}{\alpha r}\subseteq f(\ball{\bar x}{r}),
 \quad\forall r\in (0,r_*],
\end{equation}
for $\alpha<\lop{f}{\bar x}$ and for a proper $r_*>0$.
By the local injectivity of $f$ around $\bar x$, taken an arbitrary
$0<\beta<\inj{f}{\bar x}$ there exists $\delta_\beta>0$ such that
\begin{equation}     \label{in:minjfbarx}
  \|f(x_1)-f(x_2)\|\ge\beta\|x_1-x_2\|,\quad\forall x_1,\, x_2
  \in\ball{\bar x}{\delta_\beta}.
\end{equation}
Thus, if taking $r_0=\min\{r_*,\delta_\beta\}$, from inclusion (\ref{in:acoverfbarx})
it follows that
\begin{equation}   \label{string:y=fx}
  \forall y\in\ball{f(\bar x)}{\alpha r_0}\quad \exists x\in
  \ball{\bar x}{r_0}:\ y=f(x).
\end{equation}
Such $x$ must be unique in $\ball{\bar x}{r_0}$ by virtue of
inequality  (\ref{in:minjfbarx}). Therefore, one can define as a single-valued
graphical localization of $f^{-1}$ the mapping $f^\sharp:\ball{f(\bar x)}{\alpha r_0}
\longrightarrow\R^n$ given by $f^\sharp(y)=x$, where $x$ is as
in (\ref{string:y=fx}). Furthermore, one can readily see that
$f^\sharp$ is Lipschitz continuous in $\ball{f(\bar x)}{\alpha r_0}$,
because on account on inequality (\ref{in:minjfbarx}), it holds
\begin{eqnarray*}
   \|f^\sharp(y_1)-f^\sharp(y_2)\| &=& \|x_1-x_2\|\le
        \beta^{-1}\|f(x_1)-f(x_2)\| \\
        & = & \beta^{-1}\|y_1-y_2\|,
   \quad\forall y_1,\, y_2\in\ball{f(\bar x)}{\alpha r_0}.
\end{eqnarray*}

Vice versa, assume now that $f$ is strongly regular around $\bar x$.
According to the aforementioned equivalent reformulation, fixed
an arbitrary $\kappa>\reg{f}{\bar x}$ there exist neighbourhoods
$U_\kappa$ of $\bar x$ and $V_\kappa$ of $f(\bar x)$ and
a graphical single-valued localization $f^\sharp:V_\kappa\longrightarrow
\R^n$ such that
\begin{equation}  \label{in:Lipgphlocf}
   \|f^\sharp(y_1)-f^\sharp(y_2)\|\le\kappa\|y_1-y_2\|,\quad
   \forall y_1,\, y_2\in\ V_\kappa.
\end{equation}
Since $f$ is continuous at $\bar x$, fixing $\alpha<\lop{f}{\bar x}$,
it is possible to pick $r>0$ in such a way that
\begin{equation}     \label{in:2BUkfBVk}
   \ball{\bar x}{r}\subseteq U_\kappa \ \hbox{ and }\
    f(\ball{\bar x}{\alpha r})\subseteq V_\kappa.
\end{equation}
Take arbitrary $x_1,\, x_2\in\ball{\bar x}{r}$. By virtue of the
second inclusion in (\ref{in:2BUkfBVk}), it must be $y_i=f(x_i)\in
V_\kappa$, for $i=1,\, 2$. On the other hand, by virtue of the
first inclusion in (\ref{in:2BUkfBVk}), one has
$$
  f^\sharp(y_i)=x_i=f^{-1}(y_i)\cap U_\kappa,\quad i=1,\, 2.
$$
Therefore, by taking into account inequality (\ref{in:Lipgphlocf}),
one obtains
$$
  \|x_1-x_2\|\le\kappa\|y_1-y_2\|=\kappa\|f(x_1)-f(x_2)\|,
$$
which shows that $f$ is metrically injective around $\bar x$ with
$\inj{f}{\bar x}\ge\kappa^{-1}$.
By passing to the infimum over all $\kappa>\reg{f}{\bar x}=
\lop{f}{\bar x}^{-1}$, one achieves the inequality in (\ref{in:injloprel}).
This completes the proof.
\end{proof}

One is now in a position to formulate the main result of the current
section.

\begin{theorem}[Global invertibility via strict estimators]   \label{thm:gloinvest}
Let $f:\R^n\longrightarrow\R^n$ be a continuous mapping and let
$\mu\ge 0$.
Suppose that for every $x\in\R^n$ $f$ admits a strict $\mu$-estimator
$h_x$, which is strongly regular around $x$, and such that
\begin{equation}   \label{in:gloinvestcon}
   \sigma_f=\inf_{x\in\R^n}\lop{h_x}{x}>\mu.
\end{equation}
\noindent Then,
\begin{itemize}

\item[(t)] $\exists f^{-1}:\R^n\longrightarrow\R^n$;

\item[(tt)] $f^{-1}$ is Lipschitz continuous on $\R^n$
with constant $(\sigma_f-\mu)^{-1}$.

\end{itemize}
\end{theorem}

\begin{proof}
For every $x\in\R^n$, in the light of Proposition \ref{pro:strrelchar},
$h_x$ is metrically injective and, by virtue of condition (\ref{in:gloinvestcon}),
taking into account inequality (\ref{in:injloprel}) one has
$\inj{h_x}{x}\ge\lop{h_x}{x}>\mu$. So $h_x$ satisfies hypothesis (ii)
of Proposition \ref{pro:gloinvest} for every $x\in\R^n$.
The same mappings play the role of a strict $\mu$-estimator of $f$
at $x$, satisfying hypothesis (iii) of
Proposition \ref{pro:gloinvest}. These facts make it possible to
invoke that result, wherefrom all the assertions
in the thesis follow.
\end{proof}

As a first comment about Theorem \ref{thm:gloinvest}, let us
observe that it can be regarded as a kind of 'globalization'
of the local invertibility condition expressed in terms of estimators
by \cite[Theorem 1E.3]{DonRoc14}. As its local counterpart,
the quantitative condition tying the parameter $\mu$ and the bound
of linear openness plays a crucial role. Roughly speaking, this condition
says that as higher is $\sigma_f$ as looser can be the estimate of $f$
provided by the family of mappings $\{h_x:\ x\in\R^n\}$. Furthermore,
the difference $\sigma_f-\mu$ takes a transparent part in
the estimate of the Lipschitz constant of $f^{-1}$.

Another point deals with the scope of Theorem \ref{thm:gloinvest}.
It offers a starting point for establishing more specific global
inversion results beyond differentiability.
An example is provided in considering the following corollary,
which is derived by replacing ${\rm C}^1$
smoothness with existence of strict first-order approximations.
Its proof follows at once on account of Example \ref{ex:strapprox}.

\begin{corollary}
Let $f:\R^n\longrightarrow\R^n$ be a continuous mapping.
If for every $x\in\R^n$ $f$ admits a strict first-order
approximation $\Apr{f}{x}{\cdot}$ at $x$, which is strongly
regular around $x$, with
$$
  \alpha_f=\inf_{x\in\R^n}\lop{\Apr{f}{x}{\cdot}}{x}>0,
$$
then
\begin{itemize}

\item[(t)] $\exists f^{-1}:\R^n\longrightarrow\R^n$;

\item[(tt)] $f^{-1}$ is Lipschitz continuous on $\R^n$
with constant $\alpha_f^{-1}$.

\end{itemize}
\end{corollary}

\vskip1cm

\begin{example}   \label{ex:gloinvnotlip}
Consider the function $f:\R\longrightarrow\R$, defined by
$$
  f(x)=
  \left\{\begin{array}{ll}
  \sgn(x)\sqrt{|x|}, & \quad \hbox{ if } |x|\le 1, \\
  \\
  x, & \quad \hbox{ if } |x|>1,
  \end{array}
  \right.  \quad\hbox{ where }\quad
  \sgn(x)=
  \left\{\begin{array}{rl}
  -1, & \ \hbox{ if } x<0, \\
    0, & \ \hbox{ if } x=0, \\
    1, & \ \hbox{ if } x>0.
  \end{array}
  \right.
$$
Such a function is continuous in $\R$, but fails to be locally Lipschitz
in $\R$ because, as one readily sees, inequality (\ref{in:Lipcondef})
does not hold in any neighbourhood of $0$. Moreover, $f$ lacks of
differentiability at $x=0$ and at $x=\pm 1$. These pathological
features make $f$ to fall out from the scope of application of Theorem
\ref{thm:Hadamard} and Theorem \ref{thm:Pourciau}. Nevertheless, by
direct inspection one can check that $f$ is globally invertible, with
$f^{-1}:\R\longrightarrow\R$ being given by
$$
  f^{-1}(y)=
  \left\{\begin{array}{ll}
  \sgn(y)y^2, & \quad\ \hbox{ if } |y|\le 1, \\
  \\
  y, & \quad\ \hbox{ if } |y|>1.
  \end{array}
  \right.
$$
Besides, since it holds
$$
  |f^{-1}(y_1)-f^{-1}(y_2)|\le 2|y_1-y_2|,\quad\forall y_1,\, y_2\in\R,
$$
then, in contrast to $f$, function $f^{-1}$ turns out to be
Lipschitz continuous in $\R$ with constant $\ell=2$.
Let us illustrate how such an instance can be put in the framework of
global invertibility via strict estimators, where Theorem \ref{thm:gloinvest}
comes into play.
To start with, observe that $f$ admits strict $0$-estimators at each point
of $\R$. More precisely, for all those points $x\in\R$ at which $f$
is strictly differentiable one can choose as strict $0$-estimator the
affine functions  $h_x=f(x)+\Fder{f}{x}(\cdot-x)$. Thus, it results in
$$
  h_x(t)=
  \left\{\begin{array}{ll}
  f(x)+\displaystyle{t-x\over 2\sqrt{|x|}}, & \quad\ \hbox{ if } 0<|x|<1, \\
  \\
  t, & \quad\ \hbox{ if } |x|>1.
  \end{array}
  \right.
$$
For the point $x=1$, one can choose $h_1:\R\longrightarrow\R$,
given by
$$
  h_1(t)=1+\max\left\{{t-1\over 2},\, t-1\right\}.
$$
Indeed, fixed an arbitrary $\epsilon\in(0,1)$, by setting $\delta_\epsilon
={\epsilon^2\over 1+\epsilon^2}$, one finds
$$
  |(f-h_1)(x_1)-(f-h_1)(x_2)|\le \epsilon |x_1-x_2|,\quad
  \forall x_1,\, x_2\in (1-\delta_\epsilon,1+\delta_\epsilon),
$$
so $\lip{f-h_1}{1}=0$. The above inequality is obviously true if
$x_1,\, x_2\in [1,1+\delta_\epsilon)$. In the case $x_1,\, x_2
\in (1-\delta_\epsilon,1]$, since the function $x\mapsto \sqrt{x}-[1+{1\over 2}(x-1)]$
is ${\rm C}^1$ in an open set containing $[1-\delta_\epsilon,1]$,
by the mean-value theorem one can write
\begin{eqnarray*}
   |(f-h_1)(x_1)-(f-h_1)(x_2)| &\le& \sup_{x\in [1-\delta_\epsilon,1]}
   \left|{1\over 2\sqrt{x}}-{1\over 2}\right||x_1-x_2| \\
   &=& {1\over 2} \left|{1\over\sqrt{1-\delta_\epsilon}}-1\right||x_1-x_2|   \\
   &=& {\epsilon\over 2}|x_1-x_2|, \quad\forall x_1,\, x_2\in
   (1-\delta_\epsilon,1].
\end{eqnarray*}
In the case $x_1\in (1-\delta_\epsilon,1)$ and $x_2\in [1,1+\delta_\epsilon)$,
by exploiting again the above inequality for the pair $x_1,\ 1\in
(1-\delta_\epsilon,1]$, one has
\begin{eqnarray*}
  |(f-h_1)(x_1)-(f-h_1)(x_2)| &\le& |(f-h_1)(x_1)-(f-h_1)(1)| \\
  &+& |(f-h_1)(1)-(f-h_1)(x_2)| \\
  &=&  |(f-h_1)(x_1)-(f-h_1)(1)| \le{\epsilon\over 2} |x_1-1| \\
  &\le& {\epsilon\over 2}|x_1-x_2|.
\end{eqnarray*}
Similarly, one sees that for the point $x=-1$ it is possible
to take $h_{-1}:\R\longrightarrow\R$, given by
$$
  h_{-1}(t)=-1+\min\left\{{t+1\over 2},\, t+1\right\}.
$$
Finally, for $x=0$ one can choose the strict $0$-estimator $h_0:\R\longrightarrow\R$
$$
  h_0(t)=\sgn(t)\sqrt{|t|}.
$$
One has now to show that, by virtue of the above choices, each function
$h_x$ is strongly regular around $x$, then to provide an estimate of each
$\lop{h_x}{x}$, and finally to check that condition (\ref{in:gloinvestcon})
is fulfilled, namely $\sigma_f>0$.
To this purpose, notice that, in the case $0<|x|<1$, as $h_x$ are affine
and invertible, they are strongly regular around every point.
In particular, if setting $l_x(t)={t\over 2\sqrt{x}}$, one has
\begin{equation}    \label{in:HLcond1}
  \lop{h_x}{x}=\lop{l_x}{0}=\reg{l_x}{0}^{-1}={1\over 2\sqrt{x}}
  \ge {1\over 2},\quad\forall x\in\R:\ 0<|x|<1.
\end{equation}
Analogously, in the case $|x|>1$, as a linear invertible function
$h_x$ are strongly regular around each point, with
\begin{equation}    \label{in:HLcond2}
  \lop{h_x}{x}=1\ge {1\over 2},\quad\forall x\in\R:\ |x|>1
\end{equation}
(remember Example \ref{ex:minjlinmap}).
In the case $x=1$, as $h_1$ is globally invertible, having inverse
which is Lipschitz continuous in $\R^n$ with $\lip{h^{-1}_1}{1}=2$,
it is strongly regular around $x=1$, and in the light of (\ref{eq:lipreglop})
one has
\begin{equation}    \label{in:HLcond3}
   \lop{h_1}{1}={1\over 2}.
\end{equation}
The case $x=-1$ can be treated in a similar manner, after noticing
that $\gph f$ is centrally symmetric ($f$ being an odd function).
To prove that $h_0$ is strongly regular around $0$ it suffices
to observe that $h_0^{-1}$ coincides with its graphical single-valued
localization $h_0^\sharp$ (actually defined all over $\R$), given by
$$
  h_0^\sharp(y)=\sgn(y)\cdot y^2,
$$
which is locally Lipschitz around $0$. To estimate $\lop{h_0}{0}$
it is useful to notice that, since $h_0^\sharp$ is ${\rm C}^1$ around
$0$, with $\Fder{h_0^\sharp}{y}=|2y|$, a straightforward application
of the mean-value theorem gives
$$
  \lip{ h_0^\sharp}{0}\le 2.
$$
Consequently, one obtains
\begin{equation}    \label{in:HLcond4}
   \lop{h_0}{0}\ge {1\over 2}.
\end{equation}
Inequalities (\ref{in:HLcond1}), (\ref{in:HLcond2}), (\ref{in:HLcond3})
and (\ref{in:HLcond4}) allow one to deduce that $\sigma_f\ge 1/2$.
Thus, it is possible to apply Theorem \ref{thm:gloinvest}. Its
thesis leads to conclusions that are consistent with what one
can directly check.
\end{example}


\section{Invertibility via pairs of convex compacta}    \label{Sect:3}

In order to illustrate how Theorem \ref{thm:gloinvest} may interact
with constructive elements of nonsmooth analysis, let us recall
the following notions from quasidifferential calculus.

\begin{definition}   \label{def:scalqdmap}
According to \cite[Definition 2.1]{Uder07}, a mapping $g:\R^n\longrightarrow\R^m$
is said to be {\it scalarly quasidifferentiable} (for short, scalarly
q.d.) at $\bar x\in\R^n$ if for every $w\in\R^m$ the scalar function
$\langle w,g\rangle:\R^n\longrightarrow\R$, i.e. $x\mapsto\langle w,g(x)\rangle=
w^\top g(x)$, is quasidifferentiable at $\bar x$ in the sense of Demyanov-Rubinov.
\end{definition}

On the base of \cite{DemRub95,PalRol97},
Definition \ref{def:scalqdmap} amounts to say that, for every $w\in\R^m$,
$\langle w,g\rangle$ is directionally differentiable at $\bar x$, in any direction
$v\in\R^n$, and there exist two elements $g^+_w,\, g^-_w\in\Sub(\R^n)$ such that
\begin{equation}    \label{eq:dderqdrep}
  \langle w,g\rangle'(\bar x;v)=\langle w,g'(\bar x;v)\rangle=
  g^+_w(v)-g^-_w(v),\quad\forall v\in\R^n.
\end{equation}
According to the Minkowski duality $\varsigma:\CC(\R^n)\longrightarrow\Sub(\R^n)$,
the functions $g^+_w$ and $g^-_w$ can be dually represented by means of elements
of the semigroup $(\CC(\R^n),+)$ through their support functions, namely
$$
  g^+_w(v)=\supp{v}{\partial g^+_w(\nullv)}, \qquad
  g^-_w(v)=\supp{v}{\partial g^-_w(\nullv)}
$$
(see, for instance, \cite{PalUrb02}).
It is well known that the representation in (\ref{eq:dderqdrep}) can not
be unique and hence so does the dual pair $(\partial g^+_w(\nullv),
\partial g^-_w(\nullv))\in\CC^2(\R^n)=\CC(\R^n)\times\CC(\R^n)$.
To restore the uniqueness, and with that
the preliminary condition for a well defined and effective calculus, the
following equivalence relation $\sim\ \subseteq\CC^2(\R^n)$ is employed,
which was already introduced by L. H\"ormander:
$$
  (A,B)\sim (C,D) \qquad\hbox{ if }\qquad A+D=B+C.
$$
Let us denote by $\left[\partial g^+_w(\nullv),\partial
g^-_w(\nullv)\right]_\sim$ the equivalence class containing
the pair $(\partial g^+_w(\nullv),\partial g^-_w(\nullv))$.
Thus, whenever $g$ is scalarly q.d. at $\bar x$ one can consider
the mapping $\sqdcoder{g}{\bar x}:\R^m\longrightarrow\CC(\R^n)^2/_\sim$,
which is well defined by
$$
  \sqdcoder{g}{\bar x}(w)=\left[\partial g^+_w(\nullv),\partial
  g^-_w(\nullv)\right]_\sim.
$$
The above generalized differentiation concept inherits
from $\CC(\R^n)^2/_\sim$ a rich calculus.

\begin{remark}
In view of next nonsmooth analysis constructions, it is convenient
to notice that, since for every $\lambda>0$ it holds
\begin{eqnarray*}
  \langle \lambda w,g\rangle'(\bar x;v) &=& \lambda\langle w,g\rangle'(\bar x;v) =
  \lambda g^+_w(v)-\lambda g^-_w(v) \\
  &=& \supp{v}{\lambda\partial g^+_w(\nullv)}-\supp{v}{\lambda\partial g^-_w(\nullv)},
  \quad\forall v\in\R^n,
\end{eqnarray*}
and, by known calculus rules in $\CC(\R^n)^2/_\sim$, it is
$$
  \left[\lambda\partial g^+_w(\nullv),\lambda\partial
  g^-_w(\nullv)\right]_\sim=\lambda \left[\partial g^+_w(\nullv),
  \partial g^-_w(\nullv)\right]_\sim,
$$
then it results in
$$
  \sqdcoder{g}{\bar x}(\lambda w)=\lambda\sqdcoder{g}{\bar x}(w),
  \quad\forall\lambda>0,\ \forall w\in\R^m.
$$
\end{remark}

To carry out the further construction needed for the present analysis,
recall that, given two subsets $A,\, B\in\CC_0(\R^n)=\CC(\R^n)
\cup\{\varnothing\}$, the operation  $\stardif:\CC_0(\R^n)\times
\CC_0(\R^n)\longrightarrow\CC_0(\R^n)$ defined by
$$
  A\stardif B=\{x\in\R^n\ |\ x+B\subseteq A \}
$$
is known as {\it $\stardif$} (or {\it Pontryagin}) {\it -difference}
of compact convex sets (see, for instance, \cite{RubVla00}).
As other difference operations employed in nonsmooth analysis
(e.g. the Demyanov's difference \cite{RubVla00}),
actually it can be regarded as a particular instance of a more general
approach to defining algebraic operations over elements of $\CC(\R^n)$
(see, for more details, \cite{PalUrb02,Rubi92,RubVla00}).
One is now in a position to introduce the tool to be used
in formulating a global invertibility condition for nonsmooth
mappings.

\begin{definition}     \label{def:starcoqdder}
Given a mapping  $g:\R^n\longrightarrow\R^m$, suppose that $g$
is scalarly q.d. at $\bar x\in\R^n$. The set-valued mapping
$\starcoder{g}{\bar x}:\R^m\rightrightarrows\R^n$, defined as being
\begin{equation}  \label{eq:defstarcoder}
   \starcoder{g}{\bar x}(w)=\partial g^+_w(\nullv)\stardif
   \partial g^-_w(\nullv),
\end{equation}
is called {\it $\stardif$-coquasiderivative} of $g$ at $\bar x$.
\end{definition}

As a first comment to Definition \ref{def:starcoqdder} let us
remark that the equality in (\ref{eq:defstarcoder}) singles out
a uniquely defined derivative object (in $\CC_0(\R^n)$), which
is independent of the representation of the scalarized
coquasiderivative. This happens because the $\stardif$-difference is an
operation which turns out to be invariant with respect to the
equivalence relation $\sim$, namely
$$
  (A,B)\sim (C,D) \qquad\hbox{ implies }\qquad A\stardif B=C\stardif D
$$
(see, for instance, \cite[Chapter 5]{RubVla00}).
Another feature to be pointed out is that, as a set-valued mapping,
$\starcoder{g}{\bar x}$ turns out to be p.h..

\vskip.5cm

\noindent {\bf Condition $\Ccond$} A mapping $g:\R^n\longrightarrow\R^m$,
which is scalarly q.d. in a neighbourhood $U$ of $\bar x$, is said to
satisfy {\it condition $\Ccond$} around $\bar x$ provided that the
set-valued mapping $G_\Usfer:U\rightrightarrows\R^n$, given by
$$
  G_\Usfer(x)=\starcoder{g}{x}(\Usfer)=\bigcup_{v\in\Usfer}
  \starcoder{g}{x}(v),
$$
is u.s.c. at $\bar x$, i.e. for every open set $O\supseteq
\starcoder{g}{\bar x}(\Usfer)$ there exists $\delta_O>0$
such that
$$
  \starcoder{g}{x}(\Usfer)\subseteq O,\quad\forall x\in
  \ball{\bar x}{\delta_O}.
$$
The above elements enable one to establish the following sufficient
condition for the linear openness around $\nullv$ of the class
of those continuous p.h. mappings with scalarizations in 
$\DCH(\R^n)$ (for short, scalarly d.s.),
which may be of independent interest.

\begin{proposition}[Linear openness of scalarly d.s. mappings]  \label{pro:scadslop}
Let $h:\R^n\longrightarrow\R^m$ be a p.h. mapping.
Suppose that:
\begin{itemize}

\item[(i)] for every $w\in\R^m$, $\langle w,h\rangle\in\DCH(\R^n)$;

\item[(ii)] $h$ satisfies condition $\Ccond$ around $\nullv$;

\item[(iii)] $\flat^{\stardif}_0(h)=\dist{\nullv}{\starcoder{h}{\nullv}(\Usfer)}>0$.
\end{itemize}

\noindent Then, $h$ is linearly open around $\nullv$ and
\begin{equation}
  \lop{h}{\nullv}\ge \flat^{\stardif}_0(h).
\end{equation}
\end{proposition}

\begin{proof}
A short way of proving this proposition is to apply \cite[Theorem 4.1]{Uder07},
which provides sufficient conditions for the metric regularity of scalarly
q.d. mappings in terms of $\stardif$-coquasiderivative.
To see how, let us start with observing that, as $h$ acts
between finite-dimensional Euclidean spaces, the assumption on the Fr\'echet
smooth renorming of the range space as well as the assumption on the trustwortiness
of the domain space are automatically fulfilled.
Then, notice that, by virtue of hypothesis (i), $h$ is in particular scalarly
q.d. at each point of $\R^n$. Moreover, the former property entails that
$h$ is continuous in $\R^n$ (in particular, each component of $h$ is
an element of $\DCH(\R^n)$).

Now, by virtue of condition $\Ccond$, fixed an arbitrary $\epsilon
\in (0,\flat^{\stardif}_0(h))$, for a proper $\delta_\epsilon>0$
one has
$$
  \starcoder{h}{x}(\Usfer)\subseteq\oball{\starcoder{h}{\bar x}
  (\Usfer)}{\epsilon},\quad\forall x\in\ball{\bar x}{\delta_\epsilon}.
$$
As a consequence, one obtains
$$
  \inf\{\dist{\nullv}{\starcoder{h}{x}(\Usfer)}\ |\
  x\in\ball{\bar x}{\delta_\epsilon}\}\ge\flat^{\stardif}_0(h)
  -\epsilon>0.
$$
The last inequality ensures that also hypothesis (1) in Theorem 4.1
is fulfilled. Thus $h$ is metrically regular/linearly open around
$\nullv$. Furthermore, a perusal of the proof of \cite[Theorem 3.2]{Uder07},
on which Theorem 4.1 relies, reveals that it holds
$$
  \lop{h}{\nullv}\ge \flat^{\stardif}_0(h)-\epsilon.
$$
The arbitrariness of $\epsilon$ enables one to achieve the estimate
in the thesis.
\end{proof}

By combining Proposition \ref{pro:scadslop} and Theorem \ref{thm:gloinvest}
it is possible to achieve the following sufficient condition for the
global invertibility of a special class of nonsmooth mappings.

\begin{theorem}[Global invertibility via $\stardif$-difference of convex compacta]  \label{thm:gloinvconvcomp}
Let $f:\R^n\longrightarrow\R^n$ be a continuous mapping and $\mu\ge 0$.
Suppose that:
\begin{itemize}

\item[(i)] for every $x\in\R^n$, $f$ admits a strict $\mu$-estimator
of the form $\widetilde{h}_x=f(x)+h_x(\cdot-x)$;

\item[(ii)] for every $w\in\R^n$, $\langle w,h_x\rangle\in\DCH(\R^n)$;

\item[(iii)] for every $x\in\R^n$, $h_x$ satisfies condition $\Ccond$ around $\nullv$;

\item[(iv)] $\flat^{\stardif}_f=\inf_{x\in\R^n}\dist{\nullv}{\starcoder{h_x}{\nullv}(\Usfer)}>\mu$;

\item[(v)] for every $x\in\R^n$, $\inj{h_x}{\nullv}>\mu$.
\end{itemize}

\noindent Then,
\begin{itemize}

\item[(t)] $\exists f^{-1}:\R^n\longrightarrow\R^n$;

\item[(tt)] $f^{-1}$ is Lipschitz continuous on $\R^n$
with constant $(\flat^{\stardif}_f-\mu)^{-1}$.

\end{itemize}
\end{theorem}

\begin{proof}
Fix $x\in\R^n$. Owing to assumptions (ii), (iii) and (iv),
the mapping $h_x$ is linearly open around $\nullv$ with
$\lop{h_x}{\nullv}\ge\flat^{\stardif}_0(h_x)$. Consequently,
as the Euclidean distance in $\R^n$ is invariant under translation,
the mapping $\widetilde{h}_x$ is linearly open around $x$,
with $\lop{\widetilde{h}_x}{x}=\lop{h_x}{\nullv}\ge\flat^{\stardif}_0(h_x)$.
For the same reason, $\widetilde{h}_x$ is metrically injective
around $x$, with $\inj{\widetilde{h}_x}{x}=\inj{h_x}{\nullv}>\mu$
by hypothesis (v).
Thus, according to Proposition \ref{pro:strrelchar}, $\widetilde{h}_x$
turns out to be strongly regular around $x$. Moreover, by hypothesis
(iv), one has
$$
  \inf_{x\in\R^n}\lop{\widetilde{h}_x}{x}\ge\inf_{x\in\R^n}
  \flat^{\stardif}_0(h_x)=\flat^{\stardif}_f>\mu,
$$
which says that also condition (\ref{in:gloinvestcon}) is satisfied.
Thus, the thesis follows from Theorem \ref{thm:gloinvest}.
\end{proof}

\begin{remark}
The reader who remembers Remark \ref{rem:strmuestpro}(i) will
observe that, because of hypotheses (i) and (ii), $f$ is implicitly
supposed to be locally Lipschitz around each $x\in\R^n$.
Therefore Theorem \ref{thm:gloinvconvcomp} is less general than analogous
existent results, which are expressed by constructions allowing for a more
general approach (for instance, in the case of pseudo-Jacobians,
consider \cite[Corollary 3.10]{JaLaMa19}). Nonetheless, utmost
generality is not the feature aimed at in Theorem \ref{thm:gloinvconvcomp}.
Instead, it focuses on a special class of possibly nondifferentiable
mappings, admitting approximations with a special structure.
Such a structure paves the way to the employment of all the benefits
given by the calculus in $\CC^2(\R^n)/_\sim$, when hypothesis (iv)
must be checked. Take into account that $\starcoder{h_x}{\nullv}(\Usfer)
=\partial (h_x)^+_w(\nullv)\stardif\partial (h_x)^-_w(\nullv)$ is expected
to be computed easily enough, as $h_x$ is p.h..
\end{remark}


\section{Global invertibility via regular coderivatives}      \label{Sect:4}

Theorem \ref{thm:AruZhu} enables one to formulate further conditions
for global invertibility of nonsmooth mappings, even in the lack of local
Lipschitz continuity, without a direct employment of Theorem \ref{thm:gloinvest}.
This can be seen by selecting adequate nonsmooth analysis tools.
An highly successful approach to generalized differentiability relies
on the geometry of normals and graphical differentiation, which was
initiated in \cite{Mord76}. It revealed to be effective in providing
characterizations for those properties discussed in Section \ref{Sect:2}
that are fundamental for the approach at the issue.
Let us briefly recall the elements needed for the present analysis.

Given a mapping $f:\R^n\longrightarrow\R^m$ and $\bar x\in\R^n$,
the regular coderivative (a.k.a. prederivative) of $f$ at $\bar x$
is the set-valued mapping $\Rcoder{f}{\bar x}:\R^m\rightrightarrows\R^n$
defined by
$$
  \Rcoder{f}{\bar x}(v)=\{u\in\R^n \ |\ (u,-v)\in
  \RNcone{(\bar x,f(\bar x))}{\gph f}\},\quad v\in\R^m.
$$
Here, given $W\subseteq\R^p$ and $\bar w\in\ W$, the subset
$$
  \RNcone{\bar w}{W}=\left\{v\in\R^p\ \biggl|\ \limsup_{w\in W\atop w\to\bar w}
  \left\langle v,{w-\bar w\over \|w-\bar w\|}\right\rangle\le 0\right\}
$$
stands for the regular normal cone to $W$ at $\bar w$.  Whenever $f$ is
Fr\'echet differentiable at $\bar x$ its regular coderivative
becomes single-valued, taking the form $\Rcoder{f}{\bar x}(v)=
\{\Fder{f}{\bar x}^*[v]\}$ (see \cite[Theorem 1.38]{Mord06}).
In the case in which $f$ is locally Lipschitz around $\bar x$, the following
scalarized representation of the regular coderivative is valid
(see \cite[Exercise 1.70(ii)]{Mord18}
$$
  \Rcoder{f}{\bar x}(v)=\Rsubd\langle w,f\rangle(\bar x),
  \quad\forall  w\in\R^m,
$$
where, given a function $\varphi:\R^n\longrightarrow\R\cup
\{\mp\infty\}$ and $\bar x\in\dom\varphi$, the set
$$
  \Rsubd\varphi(\bar x)=\left\{v\in\R^n\ \biggl|\
  \liminf_{x\to\bar x}{\varphi(x)-\varphi(\bar x)-\langle v,
  x-\bar x\rangle\over\|x-\bar x\|}\ge 0\ \right\}
$$
is the regular subdifferential of $\varphi$ at $\bar x$
(see \cite[Chapter 1.3.4]{Mord18}).

In combination with the above nonsmooth analysis constructions,
the following monotonicity property for set-valued mappings plays
a crucial role in the present context.

\begin{definition}    \label{def:lochypmonot}
A set-valued mapping $F:\R^n\rightrightarrows\R^n$ is said to be
{\it locally hypomonotone} around $(\bar x,\bar y)\in\gph F$ if
there exist a neighbourhood $U\times V$ of $(\bar x,\bar y)$
and a constant $\gamma>0$ such that
$$
  \langle y_1-y_2,x_1-x_2\rangle\ge -\gamma\|x_1-x_2\|,\quad
  \forall (x_1,y_1),\, (x_2,y_2)\in\gph F\cap (U\times V).
$$
\end{definition}

As a comment to Definition \ref{def:lochypmonot}, it is worth mentioning
that all globally hypomonotone set-valued mappings, i.e.
set-valued mappings  $F:\R^n\rightrightarrows\R^n$ such that
$F+r\,\id$, where $\id$ stands for the identity operator, is monotone
on $\R^n$ for some $r>0$, are in particular locally hypomonotone. Moreover,
all locally monotone as well as Lipschitz continuous single-valued
mappings are locally hypomonotone (see \cite[Chapter 5]{Mord18}).
Local hypomonotonicity in synergy with a positive-definiteness
condition expressed in terms of regular coderivative is known
to yield strong metric regularity. As a consequence, the following
global invertibility condition can be established via regular
coderivatives.

\begin{theorem}[Global invertibility via regular coderivatives]
Given a mapping $f:\R^n\longrightarrow\R^n$, suppose that
\begin{itemize}

\item[(i)] $f$ is continuous in $\R^n$;

\item[(ii)] $f$ is locally hypomonotone around each $x\in\R^n$;

\item[(iii)] for every $x\in\R^n$ there exist positive
constants $\widehat{\alpha}_x$ and $\eta_x$ such that
\begin{equation}    \label{in:Rcoderposcon}
   \langle u,v\rangle\ge\widehat{\alpha}_x\|v\|^2,\quad\forall
   u\in\Rcoder{f}{z}(v),\ \forall z\in\ball{x}{\eta_x};
\end{equation}

\item[(iv)] $\widehat{\alpha}=\displaystyle\inf_{x\in\R^n}\widehat{\alpha}_x>0$.

\end{itemize}
Then
\begin{itemize}

\item[(t)] $\exists f^{-1}:\R^n\longrightarrow\R^n$;

\item[(tt)] $f^{-1}$ is Lipschitz continuous on $\R^n$
with constant $\widehat{\alpha}$.

\end{itemize}
\end{theorem}

\begin{proof}
Observe that the continuity of $f$ at each $x\in\R^n$ (hypothesis (i))
ensures the validity of the inequality in (\ref{in:Rcoderposcon})
for every $u\in\Rcoder{f}{z}(v)$ and every $(z,f(z))\in\gph f\cap
\ball{(x,f(x))}{\eta_x}$, up to a proper reduction in the value
of $\eta_x>0$. The latter condition, along with the hypothesis of
local hypomonotonicity of $f$ around $x$, is sufficient to guarantee
that the mapping $f$ is strongly regular around each $x$, with
$\lop{f}{x}\ge\widehat{\alpha}_x\ge\widehat{\alpha}$, as stated
in \cite[Corollary 5.15]{Mord18}. By virtue of such a property,
$f$ turns out to be metrically injective and hence locally injective
around each $x$, as well as $\widehat{\alpha}$-covering at each
$x$. Thus the thesis follows from Theorem \ref{thm:AruZhu}.
\end{proof}

\begin{remark}
By replacing the regular coderivative with the basic (a.k.a. Mordukhovich)
coderivative, in the particular case in which $f$ is locally Lipschitz it is possible
to obtain a point-based version of condition (\ref{in:Rcoderposcon}).
Indeed, in such a setting, by applying the scalarized representation valid
for  basic coderivatives (see \cite[Theorem 1.32]{Mord18}),
condition (\ref{in:Rcoderposcon}) can be reformulated as follows:
for every $x\in\R^n$ there exists a positive constant $\alpha_x$ such that
$$
  \langle w,v\rangle\ge\alpha_x,\quad\forall v\in\Bsubd\langle w,f\rangle
  (x),\quad\forall w\in\R^n\backslash\{\nullv\},
$$
where $\Bsubd\varphi(\bar x)=\Uplim_{x\stackrel{\varphi}{\to} \bar x}
\ \Rsubd\varphi(x)$ denotes the basic
subdifferential of $\varphi$ at $\bar x\in\dom\varphi$ and
where $\Uplim_{x\stackrel{\varphi}{\to} \bar x}$ denotes the
Painlev\'e-Kuratowski upper limit of the set-valued mapping
$\Rsubd\varphi:\R^n\rightrightarrows\R^n$ as $x\to\bar x$ and
$\varphi(x)\to\varphi(\bar x)$ (see \cite[Theorem 5.16]{Mord18}
and \cite[Section 5.4]{Mord18}).
\end{remark}


\section{Conclusions} \label{Sect:5}

The investigations exposed in the present paper starts with a recent
result about global inversion of functions in a finite-dimensional Euclidean space
setting. In order to develop its potential, its assumptions are
interpreted in the light of some specific nonsmooth analysis constructions.
The main findings of the paper provide evidences that generalization
of Hadamard's inversion theorem to nonsmooth functions not only are possible
(what was already known),
but they can be achieved through different paths, leading to a variety 
of conditions. The specific nonsmooth analysis
constructions here considered play a certain role not only in formulating
conditions for global invertibility, but also in providing quantitative
estimates for the Lipschitz constant of the inverse, in the spirit
of Pourciau's theorem.

\vskip2cm



\begin{thebibliography}{99}



\bibitem{AruZhu19} Arutyunov A.V., Zhukovskiy S.E.:
\textit{Hadamard's theorem for mappings with relaxed smoothness conditions},
Sb. Math. \textbf{210} (2019), no. 2, 165--183.

\bibitem{Clar76} Clarke F.H.:
\textit{On the inverse function theorem}, Pacific J. Math. \textbf{64}
(1976), no.1, 97--102.

\bibitem{DemRub95} Demyanov V.F., Rubinov A.M.:
\textit{Constructive Nonsmooth Analysis}, Peter Lang,
Frankfurt am Main, 1995.


\bibitem{DonRoc14} Dontchev A.L., Rockafellar R.T.:
\textit{Implicit functions and solution mappings. A view from variational
analysis}, Springer Monographs in Mathematics. Springer, Dordrecht, 2014.

\bibitem{Elcr82} Elcrat A.:
\textit{Some applications of Hadamard's inverse function theorem},
Nonlinear phenomena in mathematical sciences (Arlington, Tex., 1980),
363--369, Academic Press, New York, 1982. 

\bibitem{FacPan03} Facchinei F., Pang J.-S.:
\textit{Finite-dimensional variational inequalities and complementarity
problems. Vol. I}, Springer-Verlag, New York, 2003.

\bibitem{Gutu22} Gut\'u O.:
\textit{Global inversion for metrically regular mappings between Banach spaces},
Rev. Mat. Complut. \textbf{35} (2022), no. 1, 25--51.

\bibitem{GutJar19} Gut\'u O., Jaramillo J.A.:
\textit{Surjection and inversion for locally Lipschitz maps between Banach spaces},
J. Math. Anal. Appl. \textbf{478} (2019), no. 2, 578--594.

\bibitem{Hada06} Hadamard J.:
\textit{Sur les transformations ponctuelles},
Bull. Soc. Math. France \textbf{34} (1906), 71--84.


\bibitem{JaLaMa19} Jaramillo J.A., Lajara S., Madiedo \'O.:
\textit{Inversion of nonsmooth maps between Banach spaces},
Set-Valued Var. Anal. \textbf{27} (2019), no. 4, 921--947.

\bibitem{JeyLuc08} Jeyakumar V., Luc D. T.:
\textit{Nonsmooth vector functions and continuous optimization},
Springer, New York, 2008.

\bibitem{Mord76}  Mordukhovich B.S.:
\textit{Maximum principle in the problem of time optimal response
with nonsmooth constraints}, J. Appl. Math. Mech. \textbf{40} (1976),
no. 6, 960--969.

\bibitem{Mord06}  Mordukhovich B.S.:
\textit{Variational Analysis and Generalized Differentiation I:
Basic Theory}, Springer, Berlin, 2006.

\bibitem{Mord18} Mordukhovich B.S.:
\textit{Variational analysis and applications}, Springer Monogr. Math.
Springer, Cham, 2018.

\bibitem{PalRol97} Pallaschke D., Rolewicz R.:
{\it  Foundations of Mathematical Optimization. Convex Analysis without Linearity},
Kluwer Academic Publishers Group, Dordrecht, 1997.

\bibitem{PalUrb02} Pallaschke D., Urba\'nski R.,
\textit{Pairs of compact convex sets}, Kluwer Academic Publishers,
Dordrecht, 2002.

\bibitem{Pour82} Pourciau B.H.:
\textit{Hadamard's theorem for locally Lipschitzian maps},
J. Math. Anal. Appl. \textbf{85} (1982), no. 1, 279--285.

\bibitem{Robi91} Robinson S.M.:
\textit{An implicit-function theorem for a class of nonsmooth functions},
Math. Oper. Res. \textbf{16} (1991), no. 2, 292--309.

\bibitem{Rubi92} Rubinov A.M.:
\textit{Differences of convex compact sets and their applications
in nonsmooth analysis}, Nonsmooth optimization: methods and applications
(Erice, 1991), 366--378, Gordon and Breach, Montreux, 1992.

\bibitem{RubVla00} Rubinov A.M., Vladimirov A.A.:
\textit{Differences of convex compacta and metric spaces of convex
compacta with applications: a survey}, Nonconvex Optim. Appl., 43
Kluwer Academic Publishers, Dordrecht, 2000, 263--296.

\bibitem{Uder07} Uderzo A.:
\textit{Convex difference criteria for the quantitative stability of
parametric quasi-differentiable systems}, Set-Valued Anal. \textbf{15}
(2007), no. 1, 81--104.



\end{thebibliography}
\end{document}